\documentclass[11pt]{article}
\usepackage{amscd}
\usepackage{amsfonts}
\usepackage{amsmath}
\usepackage{amssymb}
\usepackage{amsthm}
\usepackage{bbm}
\usepackage{CJK}
\usepackage{fancyhdr}
\usepackage{graphicx}
\usepackage{indentfirst}
\usepackage{latexsym,bm}
\usepackage{mathrsfs}
\usepackage{xypic}
\usepackage[square, comma, sort&compress, numbers]{natbib}
\usepackage[top=1in,bottom=1in,left=1.25in,right=1.25in]{geometry}
\newtheorem{theorem}{Theorem}[section]

\newtheorem{definition}[theorem]{Definition}
\newtheorem{proposition}[theorem]{Proposition}
\newtheorem{example}[theorem]{Example}

\newtheorem{remark}[theorem]{Remark}
\allowdisplaybreaks[4]
\def\<{\langle}
\def\>{\rangle}
\def\a{\alpha}
\def\b{\beta}

\def\c{\cdot}

\def\d{\delta}
\def\D{\Delta}

\def\g{\gamma}

\def\o{\otimes}

\def\r{\rho}

\def\tr{\triangleright}
\def\tl{\triangleleft}

\date{}
\begin{document}

\renewcommand{\baselinestretch}{1.2}
\renewcommand{\arraystretch}{1.0}

\title{A New approach to the construction of braided T-categories}

\date{}

\author {{ Daowei Lu\textsuperscript{1}\footnote {Email: ludaowei620@126.com}
 \quad Miman You\textsuperscript{2}\footnote{Corresponding author: youmiman@126.com}  
 }\\
  {\small 1. Department of Mathematics, Jining University,}\\
 {\small Qufu Shandong 273155, P. R. of China}\\
 {\small 2. School of Mathematics and Information Science,}\\
 {\small North China University of Resources of Electric Power,}\\
 {\small Henan Zhengzhou 450045, P. R. of China}\\}

 \maketitle
\begin{center}
\begin{minipage}{12.cm}

\noindent{\bf Abstract.} The aim of this paper is to construct a new braided $T$-category via the generalized Yetter-Drinfel'd modules
and Drinfel'd codouble over Hopf algebra, an approach different from that proposed by Panaite and Staic \cite{PS}.
Moreover, in the case of finite dimensional,  we will show that this category coincides with the corepresentation of
 a certain coquasitriangular Turaev group algebra that we construct.
Finally we apply our theory to the case of group algebra.
\\

\\

 \end{minipage}
 \end{center}
 \normalsize\vskip1cm

\section*{Introduction}

  Braided $T$-categories introduced by Turaev \cite{T2} are of interest due to their applications in homotopy quantum field theories,
 which are generalizations of ordinary topological quantum field theories.
 Braided $T$-category gives rise to 3-dimensional homotopy quantum field theory and plays a key role in the construction
 of Hennings-type invariants of flat group-bundles over complements of link in the 3-sphere, see \cite{V}.
 As such, they are interesting to different research communities in mathematical physics (see \cite{FY1989, K2004}).

 The quantum double of Drinfel'd \cite{Drinfeld1990} is one of the most celebrated Hopf constructions,
 which associates to a Hopf algebra $H$ a quasitriangular Hopf algebra $D(H)$.
 Unlike the Hopf algebra axioms themselves, the axioms of a dual quasitriangular (coquasitriangular) Hopf algebra are not self-dual.
 Thus the axioms and ways of working with these coquasitriangular Hopf algebras look somewhat different in practice
 and so it is surely worthwhile to study and write them out explicity in dual form.
Moreover, the corepresentation category of coquasitriangular Hopf algebras can give rise to a braided monoidal category
 which is different from one coming from the representation category of quasitriangular Hopf algebras.
 It is these ideals which many authors studied these notions (cf.\cite{CWZ1996,Doi1992},
 \cite{FM1994}, \cite{LT1991}, \cite{Lu1994,Majid1998,Mon},\cite{Radford1993},
 \cite{Wang1999b,Wang2001}).

 In \cite{PS}, the authors found a wise method to construct braided $T$-category $\mathcal{YD}(H)$ over the group $G=Aut_{Hopf}(H)\times Aut_{Hopf}(H)$,
 where $H$ is a Hopf algebra. This category $\mathcal{YD}(H)$ is the disjoint union of all these categories
  $_H\mathcal{YD}^H(\alpha,\beta)$(the categories of $(\alpha, \beta)$-Yetter-Drinfeld modules) over $H$ for all $\a,\b\in Aut_{Hopf}(H)$.
  The authors also proved that, if H is finite dimensional, then $\mathcal{YD}(H)$ coincides
 with the representations of a certain quasitriangular $T$-coalgebra $DT (H) $.
  Our motivation is the following: Can we use $(\alpha, \beta)$-Yetter-Drinfeld modules and Drinfel'd codouble to construct a new braid $T$-category?
 And in the case of $H$ being finite dimensional,
 can we prove that this new braid $T$-category is isomorphic to the corepresentation category of a certain coquasitriangular Turaev group algebra?

 In this paper, we give a positive answer to the above question. The paper is organized as follows:

 In section 1, we recall the notions of braided $T$-category,  Turaev group algebra and generalized Yetter-Drinfel'd modules.
 In section 2, we introduce the diagonal crossed coproduct $H^{*op}\bowtie C$, where $H$ is a Hopf algebra and $C$ is an $H$-bimodule coalgebra.
 In section 3, we firstly recall the definition of $(\a,\b)$-Yetter-Drinfel'd module, then we construct braided $T$-category $\widehat{\mathcal{YD}(H)}$
  over $G$ whose multiplication is $(\a,\b)*(\g,\d)=(\d\a\d^{-1}\g,\d\b)$ for all $\a,\b,\g,\d\in Aut_{Hopf}(H)$.
   We also prove that category $\widehat{\mathcal{YD}(H)}$ coincides with the corepresentation of a certain coquasitriangular
   crossed Turaev group algebra in the sense of \cite{W}.

\section{Preliminary}
\def\theequation{1.\arabic{equation}}
\setcounter{equation} {0} \hskip\parindent

Throughout this paper, let $k$ be a fixed field, and all vector spaces and tensor product are over $k$. All vector spaces are assumed to be finite dimensional, although it should be clear when this restriction is not necessary.

In this section we recall some basic definitions and results related to our paper.

\subsection{Crossed $T$-category}

Let $G$ be a group with the unit 1. Recall from \cite{LW} that a crossed category $\mathcal{C}$ (over $G$) is given by the following data:

$\bullet $ $\mathcal{C}$ is a strict monoidal category.

$\bullet $ A family of subcategory $\{\mathcal{C_\a}\}_{\a\in G}$ such that $\mathcal{C}$ is a disjonit union of this family and that $U\o V\in\mathcal{C}_{\a\b}$ for any $\a,\b\in G$, $U\in\mathcal{C}_\a$ and $V\in\mathcal{C}_\b$.

$\bullet $ A group homomorphism $\varphi:G\rightarrow aut(\mathcal{C}),\b\mapsto\varphi_{_\b}$, the $conjugation$, where $aut(\mathcal{C})$ is the group of the invertible strict tensor functors from $\mathcal{C}$ to itself, such that $\varphi_{_\b}(\mathcal{C_\a})=\mathcal{C}_{\b\a\b^{-1}}$ for any $\a,\b\in G$.

We will use the left index notation in Turaev: Given $\b\in G$ and an object $V\in\mathcal{C_\a}$, the functor $\varphi_{_\b}$ will be denoted by $^{\b}(\cdot)$ or $^{V}(\cdot)$ and $^{\b^{-1}}(\cdot)$ will be denoted by $^{\overline{V}}(\cdot)$. Since $^{V}(\cdot)$ is a functor, for any object $U\in\mathcal{C}$ and any composition of morphism $g\circ f$ in $\mathcal{C}$, we obtain $^{V}id_U=id_{^{V}U}$ and $^{V}(g\circ f)=\! ^{V}g\circ\! ^{V}f$. Since the conjugation $\varphi:\pi\rightarrow aut(\mathcal{C})$ is a group homomorphism, for any $V,W\in\mathcal{C}$, we have $^{V\o W}(\cdot)=\! ^{V}(^{W}(\cdot))$ and $^{1}(\cdot)=\! ^{V}(^{\overline{V}}(\cdot))=\! ^{\overline{V}}(^{V}(\cdot))=id_{\mathcal{C}}$. Since for any $V\in\mathcal{C}$, the functor $^{V}(\cdot)$ is strict, we have $^{V}(f\o g)=\! ^{V}f\o\! ^{V}g$ for any morphism $f$ and $g$ in $\mathcal{C}$, and $^{V}(1)=1.$

A Turaev braided $G$-category is a crossed $T$-category $\mathcal{C}$ endowed with a braiding, i.e., a family of isomorphisms
$$c=\{c_{_{U,V}}:U\o V\rightarrow~ ^{V}U\o V\}_{U,V\in\mathcal{C}}
$$
obeying the following conditions:

$\bullet $ For any morphism $f\in Hom_{\mathcal{C}_\a}(U,U')$ and $g\in Hom_{\mathcal{C}_\b}(V,V')$, we have
$$(^{\a}g\o f)\circ c_{_{U,V}}=c_{_{U',V'}}\circ(f\o g),
$$

$\bullet $ For all $U,V,W\in\mathcal{C}$, we have
\begin{align}
c_{_{U\o V,W}}&=(c_{_{U,^{V}W}}\o id_V)(id_U\o c_{_{V,W}}),\\
c_{_{U,V\o W}}&=(id_{^{U}V}\o c_{_{U,W}})(c_{_{U,V}}\o id_{W}).
\end{align}

$\bullet $ For any $U,V\in\mathcal{C}$ and $\a\in G$, $\varphi_{_\a}(c_{_{U,V}})=c_{_{^{\a}U,^{\a}V}}$.

\subsection{Turaev Group Algebras}

Let $G$ be a group with unit 1. Recall from \cite{T,W} that a $G$-algebra is a family $A=\{A_\a\}_{\a\in G}$ of $k$-spaces together with a family of $k$-linear maps $m=\{m_{\a,\b}:A_\a\o A_\b\rightarrow A_{\a\b}\}_{\a,\b\in G}$ (called multiplication) and a $k$-linear map $\eta:k\rightarrow A_1$ (called unit) such that $m$ is associative in the sense that, for all $\a,\b,\g\in G$
\begin{eqnarray*}
&&m_{\a\b,\g}(m_{\a,\b}\o id)=m_{\a,\b\g}(id\o m_{\b,\g}),\\
&&m_{\a,1}(id\o\eta)=id=m_{1,\a}(\eta \o id).
\end{eqnarray*}
A Turaev $G$-algebra is a $G$-algebra $H=\{H_\a\}_{\a\in G}$ such that each $H_\a$ is a coalgebra with comultiplication $\Delta_\a$ and counit $\varepsilon_\a$; the map $\eta:k\rightarrow H_1$ and the maps $m_{\a,\b}:H_\a\o H_\b\rightarrow H_{\a\b}$ are coalgebra maps, with a family of $k$-linear maps $S=\{S_\a:H_\a\rightarrow H_{\a^{-1}}\}_{\a\in G}$ (called the antipode) such that for all $\a\in G$
$$m_{\a,\a^{-1}}(id\o S_\a)\Delta_\a=\varepsilon_\a 1=m_{\a^{-1},\a}(S_\a \o id)\Delta_\a.$$
Furthermore, a crossed Turaev $G$-algebra is a Turaev $G$-algebra with a family of
coalgebra isomorphisms $\psi=\{\psi_\b:H_\a\rightarrow H_{\b\a\b^{-1}}\}_{\b\in G}$ (called crossing), satisfying
the following conditions: for all $\a,\b,\g\in G$
\begin{itemize}
  \item [(i)] $\psi$ is multiplicative, i.e., $\psi_\a\psi_\b=\psi_{\a\b}$,
  \item [(ii)] $\psi$ is compatible with $m$, i.e., $m_{\g\a\g^{-1},\g\b\g^{-1}}(\psi_\g\o\psi_\g)=\psi_\g m_{\a,\b}$,
  \item [(iii)] $\psi$ is compatible with $\eta$, i.e., $\eta=\psi_\g\eta$,
  \item [(iv)] $\psi$ preserves the antipode, i.e., $\psi_\b S_\a=S_{\b\a\b^{-1}}\psi_\b$.
\end{itemize}

We use the Sweedler¡¯s notation for a comultiplication $\Delta_\a$ on $H_\a$: for all $h\in H_\a$
$$\Delta_\a(h)=h_1\o h_2.$$

Recall from \cite{W}, a Turaev $G$-algebra $H$ is called coquasitriangular if there exists a family of $k$-linear maps $\sigma=\{\sigma_{\a,\b}:H_\a\o H_\b\rightarrow k\}$ such that $\sigma_{\a,\b}$ is convolution invertible for all $\a,\b\in G$ and the following conditions are satisfied:
\begin{itemize}
  \item [(TCT1)] $\sigma_{\a\b,\g}(xy,z)=\sigma_{\a,\g}(x,z_2)\sigma_{\b,\g}(y,z_1)$,
  \item [(TCT2)]  $\sigma_{\a,\b\g}(x,yz)=\sigma_{\a,\b}(x_1,y)\sigma_{\b^{-1}\a\b,\g}(\psi_{\b^{-1}}(x_2),z)$,
  \item [(TCT3)]  $\sigma_{\a,\b}(x_1,y_1)y_2\psi_{\b^{-1}}(x_2)=x_1y_1\sigma_{\a,\b}(x_2,y_2)$,
  \item [(TCT4)]  $\sigma_{\a,\b}(x,y)=\sigma_{\g\a\g^{-1},\g\b\g^{-1}}(\psi_\g(x),\psi_\g(y))$.
\end{itemize}
for all $x\in H_\a,y\in H_\b,z\in H_\g$.

Note that if Turaev $G$-algebra $H$ is coquasitriangular, then $(H_1,\sigma_{1,1})$ is a coquasitriangular Hopf algebra.

\subsection{Yetter-Drinfel'd module}

Let $C$ be an $H$-bimodule coalgebra, with module structures $H\o C\rightarrow C,\ h\o c\mapsto h\c c$ and $C\o H\rightarrow C,\ c\o h\mapsto c\c h$. Recall from \cite{CMZ}, we can consider the Yetter-Drinfel'd datum $(H,C,H)$ and the Yetter-Drinfel'd category $_H\mathcal{YD}^C$, whose object $M$ is a left $H$-module (with the action $h\o m\mapsto h\c m$) and right $C$-comodule (with the coaction $m\mapsto m_{(0)}\o m_{(1)}$) such that for all $h\in H,m\in M$,
$$h_1\c m_{(0)}\o h_2\c m_{(1)}=(h_2\c m)_{(0)}\o(h_2\c m)_{(1)}\c h_1,$$
or equivalently
$$(h\c m)_{(0)}\o (h\c m)_{(1)}=h_2\c m_{(0)}\o h_3\c m_{(1)}\c S^{-1}(h_1).$$

\section{Diagonal crossed coproduct}
\def\theequation{2.\arabic{equation}}
\setcounter{equation} {0} 

As the dual of diagonal crossed product (for details, see \cite{PS}), we have the following result.

\begin{proposition}
Let $H$ be a Hopf algebra with a bijective antipode $S$, and $C$ a bimodule coalgebra with the actions $H\o C\rightarrow C,\ h\o c\mapsto h\c c$ and $C\o H\rightarrow C,\ c\o h\mapsto c\c h$. Then we have a coalgebra $H^{*op}\o C$ (denoted by $H^{*op}\bowtie C$) with the comultiplication and counit
\begin{eqnarray}
&&\bar{\Delta}(p\bowtie c)=\sum_{i,j}p_1\bowtie h_j\c c_1\c S^{-1}(h_i)\o h^ip_2h^j\bowtie c_2,\\
&&\bar{\varepsilon}(p\bowtie c)=p(1)\varepsilon(c),
\end{eqnarray}
for all $p\in H^{*op},c\in C$, where $\{h_i\}$ and $\{h^i\}$ are basis and dual basis of $H$. $H^{*op}\bowtie C$ is called diagonal crossed coproduct.
\end{proposition}

\begin{proof}
For all $p\in H^{*op},c\in C$, on one hand
\begin{align*}
&(\bar{\D}\o id)\bar{\D}(p\bowtie c)\\
&=\sum_{i,j}\bar{\D}(p_1\bowtie h_j\c c_1\c S^{-1}(h_i))\o h^ip_2h^j\bowtie c_2\\
&=\sum_{i,j,s,t}p_1\bowtie h_s\c(h_j\c c_1\c S^{-1}(h_i))_1S^{-1}(h_t)\o h^tp_2h^s\bowtie (h_j\c c_1\c S^{-1}(h_i))_2 \o h^ip_3h^j\bowtie c_2\\
&=\sum_{i,j,s,t}p_1\bowtie h_sh_{j1}\c c_1\c S^{-1}(h_t h_{i2})\o h^tp_2h^s\bowtie h_{j2}\c c_2\c S^{-1}(h_{i1}) \o h^ip_3h^j\bowtie c_3.
\end{align*}
Evaluating the first, the third and the fifth factors at $h,h',h''\in H$ respectively, we have
\begin{align*}
&\sum_{i,j,s,t}p_1(h)h_sh_{j1}\c c_1\c S^{-1}(h_t h_{i2})\o h^tp_2h^s(h')h_{j2}\c c_2\c S^{-1}(h_{i1}) \o h^ip_3h^j(h'') c_3\\
&=p_1(h)h'_3h''_{4}\c c_1\c S^{-1}(h'_1 h''_{2})\o p_2(h'_2)h''_{5}\c c_2\c S^{-1}(h''_{1}) \o p_3(h''_3) c_3\\
&=p(hh'_2h''_3)h'_3h''_{4}\c c_1\c S^{-1}(h'_1 h''_{2})\o h''_{5}\c c_2\c S^{-1}(h''_{1}) \o c_3.
\end{align*}
On the other hand
\begin{align*}
&(id\o\bar{\D})\bar{\D}(p\bowtie c)\\
&=\sum_{i,j}p_1\bowtie h_j\c c_1\c S^{-1}(h_i)\o \bar{\D}(h^ip_2h^j\bowtie c_2)\\
&=\sum_{i,j,s,t}p_1\bowtie h_j\c c_1\c S^{-1}(h_i)\o h^i_1p_2h^j_1\bowtie h_s\c c_2\c S^{-1}(h_t)\o h^th^i_2p_3h^j_2h^s\bowtie c_3.
\end{align*}
Evaluating the first, the third and the fifth factors at $h,h',h''\in H$ respectively, we have
\begin{align*}
&\sum_{i,j,s,t}p_1(h)h_j\c c_1\c S^{-1}(h_i)\o h^i_1p_2h^j_1(h')h_s\c c_2\c S^{-1}(h_t)\o h^th^i_2p_3h^j_2h^s(h'')c_3\\
&=\sum_{i,j}p_1(h)h_j\c c_1\c S^{-1}(h_i)\o h^i_1(h'_1)p_2(h'_2)h^j_1(h'_3)h''_5\c c_2\c S^{-1}(h''_1)\o h^i_2(h''_2)p_3(h''_3)h^j_2(h''_4)c_3\\
&=\sum_{i,j}p_1(hh'_2h''_3)h'_3 h''_4\c c_1\c S^{-1}(h'_1h''_2)\o h''_5\c c_2\c S^{-1}(h''_1)\o c_3.
\end{align*}
 Thus $\bar{\D}$ is coassociative. Easy to check that $\bar{\varepsilon}$ is counit. The proof is completed.
 \end{proof}

 \begin{remark}
 In particular when $C=H$ and the module action is multiplication, we can recover the Drinfel'd codouble $\widehat{D(H)}$ introduced in \cite[Proposition 10.3.14]{Mon}.
 \end{remark}

 \begin{proposition}
 Diagonal crossed coproduct $H^{*op}\bowtie C$ is a $\widehat{D(H)}$-bimodule coalgebra with structures
 \begin{eqnarray}
 &&\widehat{D(H)}\o H^{*op}\bowtie C\rightarrow H^{*op}\bowtie C,\ (p\o h)\tr(q\bowtie c)=qp\bowtie h\c c,\\
 &&H^{*op}\bowtie C\o\widehat{D(H)}\rightarrow H^{*op}\bowtie C,\ (q\bowtie c)\tl(p\o h)=pq\bowtie c\c h,
 \end{eqnarray}
 for all $p,q\in H^{*op},h\in H,c\in C$.
 \end{proposition}

 \begin{proof}
 Obviously $H^{*op}\bowtie C$ is a left $\widehat{D(H)}$-module. And for all $p,q\in H^{*op},h\in H,c\in C$,
 \begin{align*}
 \bar{\Delta}((p\o h)\tr(q\bowtie c))&=\bar{\Delta}(qp\bowtie h\c c)\\
                                    &=\sum_{i,j}q_1p_1\bowtie h_j\c (h\c c)_1\c S^{-1}(h_i)\o h^iq_2p_2h^j\bowtie (h\c c)_2\\
                                    &=\sum_{i,j}q_1p_1\bowtie h_j h_1\c c_1\c S^{-1}(h_i)\o h^iq_2p_2h^j\bowtie h_2\c c_2\\
                                    &=\sum_{i,j}q_1p_1\bowtie h_i h_1S^{-1}(h_j)h_s\c c_1\c S^{-1}(h_t)\o h^tq_2h^sh^j p_2h^i\bowtie h_2\c c_2\\
                                    &=(p\o h)_1\tr(q\bowtie c)_1\o(p\o h)_2\tr(q\bowtie c)_2.
 \end{align*}
 Thus $H^{*op}\bowtie C$ is a left $\widehat{D(H)}$-module coalgebra. Similarly one can check that $H^{*op}\bowtie C$ is also
 a right $\widehat{D(H)}$-module coalgebra. The proof is completed.
\end{proof}

\section{The construction of braided $T$-category $\widehat{\mathcal{YD}(H)}$}
\def\theequation{3.\arabic{equation}}
\setcounter{equation} {0} 

\begin{definition}\cite[Definition 2.1]{PS}
Let $H$ be a Hopf algebra and $\a,\b\in Aut_{Hopf}(H)$. An $(\a,\b)$-Yetter-Drinfel'd module over $H$ is a vector space $M$ such that $M$ is a left $H$-module and right $H$-comodule with the following compatible condition
$$h_1\c m_{(0)}\o\b(h_2)m_{(1)}=(h_2\c m)_{(0)}\o (h_2\c m)_{(1)}\a(h_1),$$
for all $h\in H,m\in M$. We denote by $_H\mathcal{YD}^H(\a,\b)$ the category of $(\a,\b)$-Yetter-Drinfel'd modules, morphisms being the $H$-linear and $H$-colinear.
\end{definition}

\begin{example}
For any Hopf algebra $H$ and $\a,\b\in Aut_{Hopf}(H)$, define $H_{\a,\b}$ as follows: $H_{\a,\b}=H$ with regular left $H$-module structure and right $H$-comodule structure given by
$$\r(h)=h_2\o\b(h_3)S^{-1}\a(h_1),$$
for all $h\in H$. Then $H_{\a,\b}\in\! _H\mathcal{YD}^H(\a,\b)$.
\end{example}

Let $\a,\b\in Aut_{Hopf}(H)$. We define an $H$-bimodule coalgebra $H(\a,\b)$ as follows: $H(\a,\b)=H$ as coalgebra with module structures
\begin{eqnarray*}
&&H\o H(\a,\b)\rightarrow H(\a,\b),\quad h\o h'\mapsto \b(h)h',\\
&&H(\a,\b)\o H\rightarrow H(\a,\b),\quad h'\o h\mapsto h'\a(h),
\end{eqnarray*}
for all $h,h'\in H$.

Now consider the Yetter-Drinfel'd datum $(H,H(\a,\b),H)$ and its Yetter-Drinfel'd category $_H\mathcal{YD}^{H(\a,\b)}$.

\begin{proposition} With the above notations, we have the relation:
$$_H\mathcal{YD}^{H(\a,\b)}=\! _H\mathcal{YD}^H(\a,\b).$$
\end{proposition}

Consider now the diagonal crossed coproduct $C(\a,\b)=H^{*op}\o H(\a,\b)$ with the comultiplication
$$\bar{\D}(p\bowtie h)=\sum_{i,j}p_1\bowtie \b(h_j)h_1S^{-1}\a(h_i)\o h^ip_2h^j\bowtie h_2,$$
for all $r\in H^{*op},h\in H$. Moreover $C(\a,\b)$ is a $\widehat{D(H)}$-bimodule coalgebra with module structures
\begin{eqnarray*}
&&\widehat{D(H)}\o H^{*op}\bowtie H(\a,\b)\rightarrow H^{*op}\o H(\a,\b),\ p\o h\o q\bowtie h'\mapsto qp\bowtie \b(h)h',\\
&&H^{*op}\bowtie H(\a,\b)\o \widehat{D(H)}\rightarrow H^{*op}\o H(\a,\b),\ q\bowtie h'\o p\o h\mapsto pq\bowtie h'\a(h).
\end{eqnarray*}

Since $H$ is finite dimensional, we have a category isomorphism $_H\mathcal{YD}^{H(\a,\b)}\cong \mathcal{M}^{H^{*op}\bowtie H(\a,\b)}$, hence $_H\mathcal{YD}^H(\a,\b)\cong \mathcal{M}^{H^{*op}\bowtie H(\a,\b)}$. The correspondence is given as follows. If $M\in\! _H\mathcal{YD}^H(\a,\b)$, then $M\in \mathcal{M}^{H^{*op}\bowtie H(\a,\b)}$ with structure
$$m_{[0]}\o m_{[1]}=\sum_{i,j}h_i\c m_{(0)}\o h^i\bowtie m_{(1)}.$$
Conversely if $M\in \mathcal{M}^{H^{*op}\bowtie H(\a,\b)}$, then $M\in\! _H\mathcal{YD}^H(\a,\b)$ with structures
\begin{eqnarray*}
&&h\c m=m_{[0]}(h\o\varepsilon)m_{[1]},\\
&&m_{(0)}\o m_{(1)}=m_{[0]}\o(\varepsilon^*\o id)m_{[1]}.
\end{eqnarray*}

\begin{proposition}
Let $H$ be a Hopf algebra and $\a,\b,\g,\d\in Aut_{Hopf}(H)$. If $M\in\! _H\mathcal{YD}^H(\a,\b)$, $N\in\! _H\mathcal{YD}^H(\g,\d)$, then $M\o N\in\! _H\mathcal{YD}^H(\d\a\d^{-1}\g,\d\b)$ with the following structures:
\begin{eqnarray*}
&&h\c(m\o n)=h_2\c m\o h_1\c n,\\
&&(m\o n)_{(0)}\o (m\o n)_{(1)}=m_{(0)}\o n_{(0)}\o\d(m_{(1)})\d\a\d^{-1}(n_{(1)}).
\end{eqnarray*}
for all $h\in H,m\in M,n\in N$.
\end{proposition}

\begin{proof}
Clearly $M\o N$ is a left $H$-module and right $H$-comodule. We need only to verify the compatible condition.
\begin{align*}
&h_1\c(m\o n)_{(0)}\o\d\b(h_2)(m\o n)_{(1)}\\
&=h_2\c m_{(0)}\o h_1\c n_{(0)}\o\d(\b(h_3)m_{(1)})\d\a\d^{-1}(n_{(1)})\\
&=(h_3\c m)_{(0)}\o h_1\c n_{(0)}\o\d((h_3\c m)_{(1)})\d\a\d^{-1}(\d(h_2)n_{(1)})\\
&=(h_3\c m)_{(0)}\o (h_2\c n)_{(0)}\o\d((h_3\c m)_{(1)})\d\a\d^{-1}((h_2\c n)_{(1)}\g(h_1))\\
&=(h_2\c (m\o n))_{(0)}\o(h_2\c (m\o n))_{(1)}\d\a\d^{-1}\g(h_1).
\end{align*}
The proof is completed.
\end{proof}

Note that if $M\in\! _H\mathcal{YD}^H(\a,\b)$, $N\in\! _H\mathcal{YD}^H(\g,\d)$ and $P\in\! _H\mathcal{YD}^H(\mu,\nu)$, then $(M\o N)\o P=M\o(N\o P)$ as an object in $_H\mathcal{YD}^H(\nu\d\a\d^{-1}\g\nu^{-1}\mu,\nu\d\b)$.

Denote $G=Aut_{Hopf}(H)\times Aut_{Hopf}(H)$, a group with multiplication
$$(\a,\b)*(\g,\d)=(\d\a\d^{-1}\g,\d\b).$$
The unit is $(id,id)$ and $(\a,\b)^{-1}=(\b^{-1}\a^{-1}\b,\b^{-1})$.

\begin{proposition}
Let $N\in\! _H\mathcal{YD}^H(\g,\d)$ and $(\a,\b)\in G$. Define $^{(\a,\b)}N=N$ as vector space with structures
\begin{eqnarray*}
&&h\rightharpoonup n=\a^{-1}\b(h)\c n,\\
&&n_{<0>}\o n_{<1>}=n_{(0)}\o\b^{-1}\d\a\d^{-1}(n_{(1)}).
\end{eqnarray*}
Then $^{(\a,\b)}N\in\! _H\mathcal{YD}^H(\b^{-1}\d\a\d^{-1}\g\a^{-1}\b,\b^{-1}\d\b)=\! _H\mathcal{YD}^H((\a,\b)*(\g,\d)*(\a,\b)^{-1}).$
\end{proposition}

\begin{proof}
Easy to see that $^{(\a,\b)}N$ is a left $H$-module and right $H$-comodule. We check the compatible condition.
\begin{align*}
&h_1\rightharpoonup n_{<0>}\o\b^{-1}\d\b(h_2)n_{<1>}\\
&=\a^{-1}\b(h_1)\c n_{(0)}\o\b^{-1}\d\b(h_2)\b^{-1}\d\a\d^{-1}(n_{(1)})\\
&=(\a^{-1}\b(h_2)\c n)_{(0)}\o\b^{-1}\d\a\d^{-1}[(\a^{-1}\b(h_2)\c n)_{(1)}\g\d\a^{-1}\b(h_1)]\\
&=(\a^{-1}\b(h_2)\c n)_{(0)}\o\b^{-1}\d\a\d^{-1}((\a^{-1}\b(h_2)\c n)_{(1)})\b^{-1}\d\a\d^{-1}\g\a^{-1}\b(h_1)\\
&=(\a^{-1}\b(h_2)\c n)_{<0>}\o(\a^{-1}\b(h_2)\c n)_{<1>}\b^{-1}\d\a\d^{-1}\g\a^{-1}\b(h_1)\\
&=(h_2\rightharpoonup n)_{<0>}\o(h_2\rightharpoonup n)_{<1>}\b^{-1}\d\a\d^{-1}\g\a^{-1}\b(h_1).
\end{align*}
The proof is completed.
\end{proof}

\begin{remark}
Let $M\in\! _H\mathcal{YD}^H(\a,\b)$, $N\in\! _H\mathcal{YD}^H(\g,\d)$ and $(\mu,\nu)\in G$. We have
$$^{(\a,\b)*(\g,\d)}N=\! ^{(\a,\b)}(^{(\g,\d)}N)$$
as an object in $_H\mathcal{YD}^H((\a,\b)*(\mu,\nu)*(\g,\d)*(\mu,\nu)^{-1}*(\a,\b)^{-1})$. and
$$^{(\mu,\nu)}(M\o N)=\! ^{(\mu,\nu)}M\o\! ^{(\mu,\nu)}N$$
as an object in $_H\mathcal{YD}^H((\mu,\nu)*(\a,\b)*(\g,\d)*(\mu,\nu)^{-1})$.
\end{remark}

\begin{proposition}
Let $M\in\! _H\mathcal{YD}^H(\a,\b)$ and $N\in\! _H\mathcal{YD}^H(\g,\d)$. Denote $^MN=\! ^{(\a,\b)}N$ as an object in $_H\mathcal{YD}^H((\a,\b)*(\g,\d)*(\a,\b)^{-1})$. Define the map
$$c_{M,N}:M\o N\rightarrow\! ^MN\o M,\quad m\o n\mapsto \a^{-1}(m_{(1)})\c n\o m_{(0)},$$
for all $m\in M,n\in N$. Then $c_{M,N}$ is $H$-linear $H$-colinear and satisfies the relations (1.1) and (1.2).
And $c_{_{^PM,^PN}}=c_{M,N}$. Moreover $c_{M,N}$ is bijective with inverse $c^{-1}_{M,N}(n\o m)=m_{(0)}\o\a^{-1}S(m_{(1)})\c n$.
\end{proposition}

\begin{proof}
We prove that $c_{M,N}$ is $H$-linear $H$-colinear. Indeed
\begin{align*}
c_{M,N}(h\c(m\o n))
&=c_{M,N}(h_2\c m\o h_1\c n)\\
&=\a^{-1}((h_2\c m)_{(1)}\a(h_1))\c n\o (h_2\c m)_{(0)}\\
&=\a^{-1}(\b(h_2)m_{(1)})\c n\o h_1\c m_{(0)}\\
&=h\c c_{M,N}(m\o n).
\end{align*}
And
\begin{align*}
&c_{M,N}(m\o n)_{(0)}\o c_{M,N}(m\o n)_{(1)}\\
&=(\a^{-1}(m_{(1)})\c n)_{<0>}\o m_{(0)(0)}\o\b((\a^{-1}(m_{(1)})\c n)_{<1>})\d\a\d^{-1}\g\a^{-1}(m_{(0)(1)})\\
&=(\a^{-1}(m_{(1)2})\c n)_{(0)}\o m_{(0)}\o\d\a\d^{-1}((\a^{-1}(m_{(1)2})\c n)_{(1)}\g\a^{-1}(m_{(1)1}))\\
&=\a^{-1}(m_{(1)1})\c n_{(0)}\o m_{(0)}\o\d(m_{(1)2})\d\a\d^{-1}(n_{(1)})\\
&=c_{M,N}((m\o n)_{(0)})\o(m\o n)_{(1)}.
\end{align*}
Furthermore
\begin{align*}
&(c_{M,^NP}\o id)(id\o c_{N,P})(m\o n\o p)\\
&=(c_{M,^NP}\o id)(m\o\g^{-1}(n_{(1)})\c p\o n_{(0)})\\
&=\a^{-1}(m_{(1)})\rightharpoonup(\g^{-1}(n_{(1)})\c p)\o m_{(0)}\o n_{(0)}\\
&=\g^{-1}\d\a^{-1}(m_{(1)})\g^{-1}(n_{(1)})\c p\o m_{(0)}\o n_{(0)}\\
&=\g^{-1}\d\a^{-1}\d^{-1}((m\o n)_{(1)})\c p\o (m\o n)_{(0)}\\
&=c_{M\o N,P}(m\o n\o p).
\end{align*}
Similarly we can prove (1.2). The proof is completed.
\end{proof}

Define $\widehat{\mathcal{YD}(H)}$ as the disjoint union of all $_H\mathcal{YD}^H(\a,\b)$ with $(\a,\b)\in G$.
If we endow $\widehat{\mathcal{YD}(H)}$ with monoidal structure given in Proposition 3.4,
then it becomes a strict monoidal category with the unit $k$ as an object in $_H\mathcal{YD}^H$ (with trivial structure).

 The group homomorphism $\psi:G\longrightarrow Aut(\widehat{\mathcal{YD}(H))},\ (\a,\b)\mapsto\psi_{(\a,\b)}$ is defined on components as
 \begin{align*}
 \psi_{(\a,\b)}: &\ \! _H\mathcal{YD}^H(\g,\d)\longrightarrow\! _H\mathcal{YD}^H((\a,\b)*(\g,\d)*(\a,\b)^{-1}),\\
               &\psi_{(\a,\b)}(N)=\! ^{(\a,\b)}N.
 \end{align*}
 and the functor acts on morphisms as identity. The braiding in $\widehat{\mathcal{YD}(H)}$ is given by the family $c=\{c_{M,N}\}$.
 Hence we have
\begin{proposition}
$\widehat{\mathcal{YD}(H)}$ is a braided T-category over $G$.
\end{proposition}

It is well known that for a Hopf algebra with a bijective antipode,
 the subcategory $_H\mathcal{YD}^H_{fd}$ of all finite dimensional objects in $_H\mathcal{YD}^H$ is rigid, i.e.,
  every object has left and right dualities. For the category $\widehat{\mathcal{YD}(H)}$, we have the following result.
\begin{proposition}
Let $M\in\! _H\mathcal{YD}^H(\a,\b)$ and suppose that $M$ is finite dimensional.
Then $M^*=Hom(M,k)$ belongs to $_H\mathcal{YD}^H(\b^{-1}\a^{-1}\b,\b^{-1})$ with
\begin{eqnarray*}
&&(h\c f)(m)=f(S^{-1}(h)\c m),\\
&&f_{(0)}(m)f_{(1)}=f(m_{(0)})\b^{-1}\a^{-1}S(m_{(1)}),
\end{eqnarray*}
for all $h\in H,m\in M$ and $f\in M^*$. Then $M^*$ is a left dual of $M$. Similarly we can define the right dual $^*M=Hom(M,k)$ of $M$ with
\begin{eqnarray*}
&&(h\c f)(m)=f(S(h)\c m),\\
&&f_{(0)}(m)f_{(1)}=f(m_{(0)})\b^{-1}\a^{-1}S^{-1}(m_{(1)}).
\end{eqnarray*}
Therefore the category $\widehat{\mathcal{YD}(H)}_{fd}$, the subcategory of $\widehat{\mathcal{YD}(H)}$ consisting of finite dimensional objects, is rigid.
\end{proposition}

\begin{proof}
First of all, $M^*$ is an object in $_H\mathcal{YD}^H(\b^{-1}\a^{-1}\b,\b^{-1})$. Indeed, obviously $M^*$ is a left $H$-module and right $H$-comodule. And
\begin{align*}
&(h_2\c f)_{(0)}(m)(h_2\c f)_{(1)}\b^{-1}\a^{-1}\b(h_1)\\
&=(h_2\c f)(m_{(0)})S(m_{(1)})\b^{-1}\a^{-1}\b(h_1)\\
&=f(S^{-1}(h_2)\c m_{(0)})\b^{-1}\a^{-1}S(m_{(1)})\b^{-1}\a^{-1}\b(h_1)\\
&=f(S^{-1}(h_2)\c m_{(0)})S(\b^{-1}\a^{-1}(\b S^{-1}(h_1)m_{(1)}))\\
&=f((S^{-1}(h_1)\c m)_{(0)})S(\b^{-1}\a^{-1}((S^{-1}(h_1)\c m)_{(1)})\b^{-1}S^{-1}(h_2))\\
&=f((S^{-1}(h_1)\c m)_{(0)})\b^{-1}(h_2)S(\b^{-1}\a^{-1}((S^{-1}(h_1)\c m)_{(1)}))\\
&=f_{(0)}(S^{-1}(h_1)\c m)\b^{-1}(h_2)f_{(1)}\\
&=(h_1\c f_{(0)})(m)\b^{-1}(h_2)f_{(1)},
\end{align*}
as required. Define maps
\begin{eqnarray*}
&&b_M:k\rightarrow M\o M^*,\quad 1\mapsto\sum_i m_i\o m^i,\\
&&d_M:M^*\o M\rightarrow k,\quad f\o m\mapsto f(m),
\end{eqnarray*}
where $\{m_i\}$ and $\{m^i\}$ are basis and dual basis of $M$.
We need to prove that $b_M$ and $d_M$ are $H$-linear. We compute
\begin{align*}
(h\c b_M(1))(m)&=(h\c\sum_i m_i\o m^i)(m)\\
               &=(\sum_i h_2\c m_i\o h_1\c m^i)(m)\\
               &=\sum_i h_2\c m_i m^i(S^{-1}(h_1)\c m)\\
               &=h_2S^{-1}(h_1)\c m\\
               &=\varepsilon(h)b_M(1)(m),
\end{align*}
and
\begin{align*}
d_M(h\c (f\o m))&=d_M(h_2\c f\o h_1\c m)\\
                &=(h_2\c f)(h_1\c m)\\
                &=f(S^{-1}(h_2)h_1\c m)\\
                &=\varepsilon(h)f(m)\\
                &=h\c d_M(f\o m).
\end{align*}
They are also $H$-colinear. Indeed,
\begin{align*}
b_M(1)_{(0)}(m)\o b_M(1)_{(1)}&=\sum_i m_{i(0)}m^i_{(0)}(m)\o\b^{-1}(m_{i(1)})\b^{-1}\a\b(m^i_{(1)})\\
                              &=\sum_i m_{i(0)}m^i(m_{(0)})\o\b^{-1}(m_{i(1)})\b^{-1}(S(m_{(1)}))\\
                              &=m_{(0)}\o\b^{-1}(m_{(1)1})S(m_{(1)2})\\
                              &=b_M(1)(m)\o1,
\end{align*}
and
\begin{align*}
d_M((f\o m)_{(0)})\o (f\o m)_{(1)}&=d_M(f_{(0)}\o m_{(0)})\o\b(f_{(1)})\a^{-1}(m_{(1)})\\
                                  &=f_{(0)}(m_{(0)})\b(f_{(1)})\a^{-1}(m_{(1)})\\
                                  &=f(m_{(0)})\a^{-1}(S(m_{(1)1})m_{(1)2})\\
                                  &=d_M(f\o m)_{(0)}\o d_M(f\o m)_{(1)}.
\end{align*}
It is straightforward to verify that

$(id_M\o d_M)(b_M\o id_M)=id_M$ and $(d_M\o id_{M^*})(id_{M^*}\o b_M)=id_{M^*}$.

Similarly we can prove that $^*M$ is a right dual of $M$. The proof is completed.
\end{proof}

Now we are in a position to construct a coquasitriangular Turaev group algebra over $G$, denoted by $CT(H)$ such that the $T$-category $Corep(CT(H))$ of corepresentation of $CT(H)$ is isomorphic to $\widehat{\mathcal{YD}(H)}$ as braided $T$-category.

For $(\a,\b)\in G$, the $(\a,\b)$-component $CT(H)_{\a,\b}$ will be the diagonal crossed coproduct $H^{*op}\bowtie H(\a,\b)$. Define multiplication by
\begin{align}
m_{(\a,\b),(\g,\d)}:&H^{*op}\bowtie H(\a,\b)\o H^{*op}\bowtie H(\g,\d)\longrightarrow H^{*op}\bowtie H((\a,\b)*(\g,\d)),\nonumber\\
&(p\bowtie h)\o(q\bowtie h')\mapsto qp\bowtie \d(h)\d\a\d^{-1}(h').
\end{align}
Then we have the following result.

\begin{proposition}
$CT(H)$ becomes a Turaev $G$-algebra under the diagonal crossed coproduct and multiplication (3.1). The antipode is given by
\begin{eqnarray*}
&&S_{(\a,\b)}:H^{*op}\bowtie H(\a,\b)\longrightarrow H^{*op}\bowtie H((\a,\b)^{-1}),\\
&&p\bowtie h\mapsto \sum_{i,j}h^iS^{-1*}(p)S^{-1*}(h^j)\bowtie\b^{-1}(h_j)\b^{-1}\a^{-1}S(h_1)\b^{-1}\a^{-1}\b(h_i).
\end{eqnarray*}

\end{proposition}

\begin{proof}
The multiplication is associative. For all $f\bowtie h\in H^{*op}\bowtie H(\a,\b),p\bowtie h'\in H^{*op}\bowtie H(\g,\d),q\bowtie h''\in H^{*op}\bowtie H(\mu,\nu)$, we compute
\begin{align*}
[(f\bowtie h)(p\bowtie h')](q\bowtie h'')&=(pf\bowtie \d(h)\d\a\d^{-1}(h'))(q\bowtie h'')\\
                                         &=qpf\bowtie\nu\d(h)\nu\d\a\d^{-1}(h')\nu\d\a\d^{-1}\g\nu^{-1}(h'')\\
                                         &=(f\bowtie h)(qp\bowtie\nu(h')\nu\g\nu^{-1}(h''))\\
                                         &=(f\bowtie h)[(p\bowtie h')(q\bowtie h'')],
\end{align*}
as claimed. Next we prove that $m_{(\a,\b),(\g,\d)}$ is a coalgebra map. Indeed,
\begin{align*}
&m_{(\a,\b),(\g,\d)}((p\bowtie h)_1\o(q\bowtie h')_1)\o m_{(\a,\b),(\g,\d)}((p\bowtie h)_2\o(q\bowtie h')_2)\\
&=\sum_{i,j,s,t} m_{(\a,\b),(\g,\d)}(p_1\bowtie\b(h_j)h_1\a S^{-1}(h_i)\o q_1\bowtie \d(h_s)h'_1\g S^{-1}(h_t))\\
&\o m_{(\a,\b),(\g,\d)}(h^ip_2h^j\bowtie h_2\o h^tq_2h^s\bowtie h'_2)\\
&=\sum_{i,j,s,t} q_1p_1\bowtie\d\b(h_j)\d(h_1)\d\a S^{-1}(h_i)\d\a(h_s)\d\a\d^{-1}(h'_1)\d\a\d^{-1}\g S^{-1}(h_t)\\
&\o h^tq_2h^sh^ip_2h^j\bowtie \d(h_2)\d\a\d^{-1}(h'_2)\\
&=\sum_{j,t} q_1p_1\bowtie\d\b(h_j)\d(h_1)\d\a\d^{-1}(h'_1)\d\a\d^{-1}\g S^{-1}(h_t)\o h^tq_2p_2h^j\bowtie \d(h_2)\d\a\d^{-1}(h'_2)\\
&=(qp\bowtie \d(h)\d\a\d^{-1}(h'))_1\o (qp\bowtie \d(h)\d\a\d^{-1}(h'))_2\\
&=m_{(\a,\b),(\g,\d)}(p\bowtie h\o q\bowtie h')_1\o m_{(\a,\b),(\g,\d)}(p\bowtie h\o q\bowtie h')_2,
\end{align*}
as required. Easy to see that $(\varepsilon\bowtie 1)_1\o(\varepsilon\bowtie 1)_2=\varepsilon\bowtie 1\o\varepsilon\bowtie 1$.

We now check that $S$ is the antipode of $CT(H)$.
\begin{align*}
&S_{(\a,\b)}((p\bowtie h)_1)(p\bowtie h)_2\\
&=\sum_{i,j}S_{(\a,\b)}(p_1\bowtie\b(h_j)h_1\a S^{-1}(h_i))(h^ip_2h^j\bowtie h_2)\\
&=\sum_{i,j,s,t} (h^sS^{-1*}(p_1)S^{-1*}(h^t)\bowtie\b^{-1}(h_t h_i)\b^{-1}\a^{-1}S(h_1)\b^{-1}\a^{-1}\b S(h_j)\b^{-1}\a^{-1}\b(h_s))(h^ip_2h^j\bowtie h_2)\\
&=\sum_{i,j,s,t}h^ip_2h^jh^sS^{-1*}(p_1)S^{-1*}(h^t)\bowtie h_th_i\a^{-1}S(h_1)\a^{-1}\b (S(h_j)h_s)\a^{-1}(h_2)\\
&=\sum_{i,j,t}h^ip_2h^jS^{-1*}(p_1)S^{-1*}(h^t)\bowtie h_th_i\a^{-1}S(h_1)\a^{-1}\b (S(h_{i1})h_{i2})\a^{-1}(h_2)\\
&=\sum_{i,j,t}h^ip_2S^{-1*}(p_1)S^{-1*}(h^t)\bowtie h_th_i\a^{-1}S(h_1)\a^{-1}(h_2)\\
&=p(1)\varepsilon(h)\varepsilon\bowtie1.
\end{align*}
Thus $S_{(\a,\b)}*id_{(\a,\b)}=\varepsilon_{(\a,\b)}\varepsilon\bowtie1$. Similarly one can verify that $id_{(\a,\b)}*S_{(\a,\b)}=\varepsilon_{(\a,\b)}\varepsilon\bowtie1$. $S$ is the antipode of $CT(H)$. The proof is completed.
\end{proof}

\begin{proposition}
Moreover $CT(H)$ is a crossed Turaev $G$-algebra with the crossing $\psi$ given by
\begin{align*}
\psi_{(\a,\b)}:&H^{*op}\bowtie H(\g,\d)\longrightarrow H^{*op}\bowtie H((\a,\b)*(\g,\d)*(\a,\b)^{-1}),\\
&p\bowtie h\mapsto p\circ\a^{-1}\b\bowtie\b^{-1}\d\a\d^{-1}(h).
\end{align*}
\end{proposition}

\begin{proof}
First of all $\psi_{(\a,\b)}$ is bijective and for all $p\in H^*,h\in H$,
\begin{align*}
&\psi_{(\a,\b)}(p\bowtie h)_1\o\psi_{(\a,\b)}(p\bowtie h)_2\\
&=(p\circ\a^{-1}\b\bowtie\b^{-1}\d\a\d^{-1}(h))_1\o(p\circ\a^{-1}\b\bowtie\b^{-1}\d\a\d^{-1}(h))_2\\
&=\sum_{i,j}p_1\circ\a^{-1}\b\bowtie\b^{-1}\d\b(h_j)\b^{-1}\d\a\d^{-1}(h_1)\b^{-1}\d\a\d^{-1}\g\a^{-1}\b S^{-1}(h_i)\o h^i(p_2\circ\a^{-1}\b)h^j\bowtie\b^{-1}\d\a\d^{-1}(h_2)\\
&=\sum_{i,j}p_1\circ\a^{-1}\b\bowtie\b^{-1}\d\a(h_j)\b^{-1}\d\a\d^{-1}(h_1)\b^{-1}\d\a\d^{-1}\g S^{-1}(h_i)\o (h^ip_2h^j)\circ\a^{-1}\b\bowtie\b^{-1}\d\a\d^{-1}(h_2)\\
&=\sum_{i,j}\psi_{(\a,\b)}(p_1\bowtie\d(h_j)h_1\g S^{-1}(h_i))\o\psi_{(\a,\b)}(h^ip_2h^j\bowtie h_2)\\
&=\psi_{(\a,\b)}((p\bowtie h)_1)\o\psi_{(\a,\b)}((p\bowtie h)_2).
\end{align*}
Thus $\psi_{(\a,\b)}$ is a coalgebra isomorphism. And

\begin{itemize}
  \item [(i)] $\psi$ is multiplicative, since for $h\in H(\mu,\nu)$
\begin{align*}
\psi_{(\a,\b)}\psi_{(\g,\d)}(p\bowtie h)&=\psi_{(\a,\b)}(p\circ\g^{-1}\d\bowtie\d^{-1}\nu\g\nu^{-1}(h))\\
                                        &=p\circ\g^{-1}\d\a^{-1}\b\bowtie\b^{-1}\d^{-1}\nu\d\a\d^{-1}\g\nu^{-1}(h)\\
                                        &=\psi_{(\d\a\d^{-1}\g,\d\b)}(p\bowtie h)\\
                                        &=\psi_{(\a,\b)*(\g,\d)}(p\bowtie h).
\end{align*}
Obviously $\psi_{(1,1)}(CT(\a,\b))=id_{(\a,\b)}$.
  \item [(ii)] For $p,q\in H^*$ and $h\in H(\g,\d),h'\in H(\mu,\nu)$,
\begin{align*}
&\psi_{(\a,\b)}(p\bowtie h)\psi_{(\a,\b)}(q\bowtie h')\\
&=(p\circ\a^{-1}\b\bowtie\b^{-1}\d\a\d^{-1}(h))(q\circ\a^{-1}\b\bowtie\b^{-1}\nu\a\nu^{-1}(h'))\\
&=qp\circ\a^{-1}\b\bowtie\b^{-1}\nu\d\a\d^{-1}(h)\b^{-1}\nu\d\a\d^{-1}\g\nu^{-1}(h')\\
&=qp\circ\a^{-1}\b\bowtie\b^{-1}\nu\d\a\d^{-1}\nu^{-1}(\nu(h)\nu\g\nu^{-1}(h'))\\
&=\psi_{(\a,\b)}(qp\bowtie\nu(h)\nu\g\nu^{-1}(h'))\\
&=\psi_{(\a,\b)}((p\bowtie h)(q\bowtie h')).
\end{align*}
  \item [(iii)] $\psi_{(\a,\b)}(\varepsilon\bowtie 1)=\varepsilon\bowtie 1$.
  \item [(iv)] \begin{align*}
&\psi_{(\a,\b)}S_{(\g,\d)}(p\bowtie h)\\
&=\sum_{i,j}\psi_{(\a,\b)}(h^iS^{-1*}(p)S^{-1*}(h^j)\bowtie\d^{-1}(h_j)\d^{-1}\g^{-1}(S(h))\d^{-1}\g^{-1}\d(h_i))\\
&=\sum_{i,j}(h^iS^{-1*}(p)S^{-1*}(h^j))\circ\a^{-1}\b\bowtie\b^{-1}\d^{-1}\a\d(\d^{-1}(h_j)\d^{-1}\g^{-1}(S(h))\d^{-1}\g^{-1}\d(h_i))\\
&=\sum_{i,j}(h^iS^{-1*}(p)S^{-1*}(h^j))\circ\a^{-1}\b\bowtie\b^{-1}\d^{-1}\a(h_j\g^{-1}(S(h))\g^{-1}\d(h_i))\\
&=\sum_{i,j}h^iS^{-1*}(p\circ\a^{-1}\b)S^{-1*}(h^j)\bowtie\b^{-1}\d^{-1}\b(h_j)\b^{-1}\d^{-1}\a\g^{-1}S(h)\b^{-1}\d^{-1}\a\g^{-1}\d\a^{-1}\b(h_i)\\
&=S_{(\a,\b)*(\g,\d)*(\a,\b)^{-1}}(p\circ\a^{-1}\b\bowtie\b^{-1}\d\a\d^{-1}(h))\\
&=S_{(\a,\b)*(\g,\d)*(\a,\b)^{-1}}\psi_{(\a,\b)}(p\bowtie h).
\end{align*}
\end{itemize}
The proof is completed.
\end{proof}

\begin{proposition}
$CT(H)$ is coquasitriangular with the structure
$$\sigma_{(\a,\b),(\g,\d)}(p\bowtie h,q\bowtie h')=p(\d^{-1}(h'))q(1)\varepsilon(h).$$
\end{proposition}

\begin{proof}
For all $f,p,q\in H^*,h\in H(\a,\b),h'\in H(\g,\d),h''\in H(\mu,\nu)$,\\
For (TCT1):
\begin{align*}
&\sigma_{(\a,\b)*(\g,\d),(\mu,\nu)}((f\bowtie h)(p\bowtie h'),(q\bowtie h''))\\
&=\sigma_{(\a,\b)*(\g,\d),(\mu,\nu)}(pf\bowtie\d(h)\d\a\d^{-1}(h'),(q\bowtie h''))\\
&=pf(\nu^{-1}(h''))q(1)\varepsilon(hh')\\
&=p(\nu^{-1}(h''_1))f(\nu^{-1}(h''_2))q(1)\varepsilon(hh'),
\end{align*}
and
\begin{align*}
&\sigma_{(\a,\b),(\mu,\nu)}(f\bowtie h,(q\bowtie h'')_2)\sigma_{(\g,\d),(\mu,\nu)}(p\bowtie h',(q\bowtie h'')_1)\\
&=\sum_{i,j}\sigma_{(\a,\b),(\mu,\nu)}(f\bowtie h,h^iq_2h^j\bowtie h''_2)\sigma_{(\g,\d),(\mu,\nu)}(p\bowtie h',q_1\bowtie \nu(h_j)h''_1\mu S^{-1}(h_i))\\
&=\sum_{i,j}f(\nu^{-1}(h''_2))h^j(1)q_2(1)h^i(1)\varepsilon(h)p(h_j\nu^{-1}(h''_1)\nu^{-1}\mu S^{-1}(h_i))\\
&=f(\nu^{-1}(h''_2))p(\nu^{-1}(h''_1))\varepsilon(hh')q(1).
\end{align*}
For (TCT2):
\begin{align*}
&\sigma_{(\a,\b),(\g,\d)*(\mu,\nu)}(f\bowtie h,(p\bowtie h')(q\bowtie h''))\\
&=\sigma_{(\a,\b),(\g,\d)*(\mu,\nu)}(f\bowtie h,qp\bowtie \nu(h')\nu\g\nu^{-1}(h''))\\
&=f(\d^{-1}(h'\g\nu^{-1}(h'')))qp(1)\varepsilon(h),
\end{align*}
and
\begin{align*}
&\sigma_{(\a,\b),(\g,\d)}((f\bowtie h)_1,p\bowtie h')\sigma_{(\g,\d)^{-1}*(\a,\b)*(\g,\d),(\mu,\nu)}(\psi_{(\g,\d)^{-1}}((f\bowtie h)_2),q\bowtie h'')\\
&=\sum_{i,j}\sigma_{(\a,\b),(\g,\d)}(f_1\bowtie \b(h_j)h_1\a S^{-1}(h_i),p\bowtie h')\\
&\quad\sigma_{(\g,\d)^{-1}*(\a,\b)*(\g,\d),(\mu,\nu)}(\psi_{(\g,\d)^{-1}}(h^if_2h^j\bowtie h_2),q\bowtie h'')\\
&=f_1(\d^{-1}(h'))p(1)\sigma_{(\g,\d)^{-1}*(\a,\b)*(\g,\d),(\mu,\nu)}(f_2\circ\d^{-1}\g\bowtie\d\b\d^{-1}\g^{-1}\d\b^{-1}(h_2),q\bowtie h'')\\
&=f_1(\d^{-1}(h'))qp(1)f_2(\d^{-1}\g\nu^{-1}(h''))\varepsilon(h)\\
&=f(\d^{-1}(h')\d^{-1}\g\nu^{-1}(h''))qp(1)\varepsilon(h).
\end{align*}
For (TCT3):
\begin{align*}
&\sigma_{(\a,\b),(\g,\d)}((f\bowtie h)_1,(p\bowtie h')_1)(p\bowtie h')_2\psi_{(\g,\d)^{-1}}((f\bowtie h)_2)\\
&=\sum_{i,j,s,t}\sigma_{(\a,\b),(\g,\d)}(f_1\bowtie \b(h_j)h_1\a S^{-1}(h_i),p_1\bowtie \d(h_s)h'_1\g S^{-1}(h_t))(h^tp_2h^s\bowtie h'_2)\psi_{(\g,\d)^{-1}}(h^if_2h^j\bowtie h_2)\\
&=\sum_{s,t}f_1(h_s\d^{-1}(h'_1)\d^{-1}\g S^{-1}(h_t))p_1(1)(h^tp_2h^s\bowtie h'_2)\psi_{(\g,\d)^{-1}}(f_2\bowtie h)\\
&=\sum_{s,t}f_1(h_s\d^{-1}(h'_1)\d^{-1}\g S^{-1}(h_t))p_1(1)(h^tp_2h^s\bowtie h'_2)(f_2\circ\d^{-1}\g\bowtie\d\b\d^{-1}\g^{-1}\d\b^{-1}(h))\\
&=\sum_{s,t}f_1(h_s\d^{-1}(h'_1)\d^{-1}\g S^{-1}(h_t))(f_2\circ\d^{-1}\g)h^tph^s\bowtie \d\b\d^{-1}(h'_2)\d(h)\\
&=\sum_{s,t}f_2(\d^{-1}(h'_1))(f_4\circ\d^{-1}\g)(f_3\circ\d^{-1}\g S^{-1})pf_1\bowtie \d\b\d^{-1}(h'_2)\d(h)\\
&=f_2(\d^{-1}(h'_1))pf_1\bowtie \d\b\d^{-1}(h'_2)\d(h),
\end{align*}
and
\begin{align*}
&(f\bowtie h)_1(p\bowtie h')_1\sigma_{(\a,\b),(\g,\d)}((f\bowtie h)_2,(p\bowtie h')_2)\\
&=\sum_{i,j,s,t}(f_1\bowtie \b(h_j)h_1\a S^{-1}(h_i))(p_1\bowtie \d(h_s)h'_1\g S^{-1}(h_t))\sigma_{(\a,\b),(\g,\d)}(h^if_2h^j\bowtie h_2,h^tp_2h^s\bowtie h'_2)\\
&=\sum_{i,j,s,t}p_1f_1\bowtie\d\b(h_j)\d(h_1)\d\a S^{-1}(h_i)\d\a(h_s)\d\a\d^{-1}(h'_1)\d\a\d^{-1}\g S^{-1}(h_t)h^if_2h^j(\d^{-1}(h'_2))\varepsilon(h_2)h^tp_2h^s(1)\\
&=pf_1\bowtie\d\b\d^{-1}(h'_4)\d(h)\d\a\d^{-1}S^{-1}(h'_2)\d\a\d^{-1}(h'_1)f_2(\d^{-1}(h'_3))\\
&=pf_1\bowtie\d\b\d^{-1}(h'_2)\d(h)f_2(\d^{-1}(h'_1)).
\end{align*}
For (TCT4):
\begin{align*}
&\sigma_{(\a,\b)*(\g,\d)*(\a,\b)^{-1},(\a,\b)*(\mu,\nu)*(\a,\b)^{-1}}(\psi_{(\a,\b)}(p\bowtie h'),\psi_{(\a,\b)}(q\bowtie h''))\\
&=\sigma_{(\a,\b)*(\g,\d)*(\a,\b)^{-1},(\a,\b)*(\mu,\nu)*(\a,\b)^{-1}}(p\circ\a^{-1}\b\bowtie\b^{-1}\d\a\d^{-1}(h'), q\circ\a^{-1}\b\bowtie \b^{-1}\nu\a\nu^{-1}(h''))\\
&=p(\nu^{-1}(h''))q(1)\varepsilon(h')\\
&=\sigma_{(\g,\d),(\mu,\nu)}(p\bowtie h',q\bowtie h'').
\end{align*}
The proof is completed.
\end{proof}

By the arguments after Proposition 3.3 we obtain the main result:
\begin{theorem}
$Corep(CT(H))$ and $\widehat{\mathcal{YD}(H)}$ are isomorphic as braided $T$-category over $G$.
\end{theorem}

\begin{example}
Let $\pi$ be a group, then we have a group algebra $k(\pi)$. It is well known that the group $Aut_{Hopf}(k(\pi))$ of Hopf automorphisms of $k(\pi)$ is equal to the group $Aut(\pi)$ of automorphisms of $\pi$. Let $\a,\b\in Aut(\pi)$. An $(\a,\b)$-Yetter-Drinfel'd module is a left $\pi$-module $M$ with a decomposition $M=\bigoplus_{a\in\pi}M_a$, where $M_a=\{m\in  M|m_{(0)}\o m_{(1)}=m\o a\}$.

If $\a,\b,\g,\d\in Aut(\pi)$, $M\in\! _{k(\pi)}\mathcal{YD}^{k(\pi)}(\a,\b)$ and $N\in\! _{k(\pi)}\mathcal{YD}^{k(\pi)}(\g,\d)$, then $M\o N\in\! _{k(\pi)}\mathcal{YD}^{k(\pi)}(\d\a\d^{-1}\g,\d\b)$ with action
$a\c(m\o n)=a\c m\o a\c n$ for all $a\in \pi,m\in M,n\in N$, and decomposition $M\o N=\bigoplus_{c\in\pi}(\bigoplus_{ab=c}M_{\d^{-1}(a)}\o N_{\d\a^{-1}\d^{-1}(h)})$.

If $\a,\b\in Aut(\pi)$ and $N\in\! _{k(\pi)}\mathcal{YD}^{k(\pi)}(\g,\d)$, then $^{(\a,\b)}N=N$ as vector space with action $a\rightharpoonup n=\a^{-1}\b(a)\c n$ for all $a\in \pi,n\in N$, and decomposition $^{(\a,\b)}N=\bigoplus_{a\in\pi}N_{\d\a^{-1}\d^{-1}\b(a)}$.

With the above notations, the braiding $c_{M,N}:M\o N\rightarrow\! ^{M}N\o M$ acts on homogeneous elements $m\in M_a,n\in N_b$ as $c_{M,N}(m\o n)=\a^{-1}(a)\c n\o m_{(0)}$. Therefore $M_\a\o N_\b$ is sent to $N_{\d\a^{-1}(a)b\g\a^{-1}(a^{-1})}\o M_a$.

Now assume that $M\in\! _{k(\pi)}\mathcal{YD}^{k(\pi)}(\a,\b)$ is finite dimensional. Since $S=S^{-1}$ for $k(\pi)$, we have $M^*=\! ^*M$, and for all $a\in\pi,m\in M, f\in M^*$, $(a\c f)(m)=f(a^{-1}\c m)$ with decomposition $M^*=\bigoplus_{a\in\pi}(M_{\b^{-1}\a^{-1}(a)})^*$.

Let $\pi$ be a finite group and $\{p_a\}_{a\in\pi}$ the dual of $k(\pi)$. For $\a,\b\in Aut(\pi)$, the component $CT(k(\pi))(\a,\b)=k(\pi)^{*op}\bowtie k(\pi)$ with comultiplication
$$\bar{\Delta}(p_c\bowtie d)=\sum_{ab=c}p_a\bowtie\b(b)d\a(b^{-1})\o p_b\bowtie d,$$
for all $c,d\in\pi$. Furthermore for $a\in k(\pi)(\a,\b)$ and $b\in k(\pi)(\g,\d)$,
\begin{eqnarray*}
&&(p_c\bowtie a)(p_d\bowtie b)=\delta_{c,d}p_c\bowtie\d(a)\d\a\d^{-1}(b),\\
&&1_{CT(k(\pi))(id,id)}=\sum_{a\in\pi}p_a\o 1,\\
&&\psi_{(\a,\b)}(p_c\bowtie d)=p_{\b^{-1}\a(c)}\o\b^{-1}\d\a\d^{-1}(d),\\
&&S_{(\a,\b)}(p_c\bowtie a)=p_{c^{-1}}\bowtie\b^{-1}(c)\b^{-1}\a^{-1}(a^{-1})\b^{-1}\a^{-1}\b(c^{-1}),\\
&&\sigma_{(\a,\b),(\g,\d)}((p_c\bowtie a),(p_d\bowtie b))=\delta_{b,\d(c)}\delta_{1,d}.
\end{eqnarray*}
\end{example}

\section*{Acknowledgement}

This work was supported by the NSF of China (No. 11371088, No. 10871042, No.11571173) and the fundamental research funds for the central universities(No. KYLX15$\_$0109).


\begin{thebibliography}{aa}

 \bibitem{CMZ} S. Caenepeel, G. Militaru and S. Zhu. Frobenius and Separable Functors for Gen-
 eralized Module Categories and Nonlinear Equations, Lecture Notes in Mathematics 1787, Springer-Verlag, Berlin, 2002.

 \bibitem{CWZ1996} M. Cohen, S. Westreich, S. Zhu. Determinants, integrality, and Noether¡¯s theorem for quantum commutative algebras.
 Israel of Math. 96(1996): 185--222.

 \bibitem{Doi1992} Y. Doi. Braided bialgebras and quadratic bialgebras. Comm. Algebra. 153(1992):1731--1749.

 \bibitem{Drinfeld1990} V. G. Drinfel'd. On almost cocommutative Hopf algebras. Leningrad Math. J. 1(1990): 321--342.


 \bibitem{FM1994} D. Fischman, S. Montgomery. A Schur¡¯s double centralizer theorem for cotriangular
  Hopf algebras and generalized Lie algebras.J. Algebra. 168(1994): 594--614.

  \bibitem{FY1989} P. J. Freyd , D. N. Yetter. Braided
 compact closed categories with applications to low-dimensional topology.
   Adv. in Math. 77(1989): 156--182.

\bibitem{K2004} A. J. Kirillov. On $G$-equivariant modular categories. (2004), Math. QA/0401119.

 \bibitem{LT1991} R. Larson, J. Towber, Two dual classes of bialgebras related to the concepts of
   quantum group and quantum Lie algebras.Comm. Algebra. 19(1991): 3295--3345.


 \bibitem{LW} L. Liu and S. Wang. Constructing new braided $T$-categories over weak Hopf algebras, Appl. Category Struct. 18(2010): 431--459.

 \bibitem{Lu1994} J. H. Lu. On the Drinfeld double and the Heisengerg double of a Hopf algebra.
   Duke Math. J. 74(1994): 763--776.

 \bibitem{Majid1998} S. Majid. Foundations of Quantum Group Theory. Cambridge Univ. Press. 1998.


 \bibitem{Mon} S. Montgomery. Hopf algebras and their actions on rings, American Mathematical Society, 1992.

 \bibitem{PS} F. Panaite, M. Staic. Generalized (anti)Yetter-Drinfeld modules as components of a braided $T$-category, Isr. J. Math., 158(2007): 349--365.

 \bibitem{Radford1993} D. E. Radford. Minimal quasitriangular Hopf algebras, J. Algebra. 157(1993): 285--315.

 \bibitem{T} V. Turaev. Homotopy quantum field theory, with appendices by M. M¨¹ger and A. Virelizier.
 In: Tracts in Math., vol. 10. European Mathematical Society, Helsinki (2010).


 \bibitem{T2} V. Turaev. Crossed group-categories, Arabian Journal for Science and
 Engineering, 33(2008): 483--503.

 \bibitem{V} A. Virelizier. Involutory Hopf group-coalgebras and flat bundles over 3-manifolds, Fund. Math.,
 188(2005): 241--270.

 \bibitem{YW} T. Yang and S. H. Wang. Constructing new braided $T$-categories over regualar multiplier Hopf algebras, Comm. Algebra,  39(2011): 3073--3089.


  \bibitem{Wang1999b} S. H. Wang, On braided Hopf algebra structure over the twisted smash products.
   Comm. Algebra. 27(1999): 5561--5573.

  \bibitem{Wang2001} S. H. Wang,  On the braided structures of bicrossproduct Hopf algebras.Tsukuba J.
  Math. 25(1)(2001): 103--120.


 \bibitem{W} S. H. Wang. New Turaev Braided Group Categories and Group Schur-Weyl Duality, Appl. Categor. Struct., 21(2013): 141--166.



\end{thebibliography}
\end{document}